\documentclass[reqno,12pt,]{amsart}
\usepackage{bbm}
\usepackage{amsmath}
\usepackage{txfonts}
\usepackage{stmaryrd}
\usepackage{amssymb}
\usepackage{amsfonts}
\usepackage{times}
\usepackage{mathrsfs}
\usepackage{amsthm}
\usepackage{enumerate}
\usepackage{framed}
\usepackage{lipsum}
\usepackage{color}

\newtheorem{theorem}{Theorem}

\newtheorem{corollary}{Corollary}

\newtheorem{lemma}{Lemma}

\def\PD{\partial\mathbb{D}}
\def\H{\mathcal{H}}
\def\DD{\mathbb{D}}

\def\rt{\right}

\def\lf{\left}

\def\tmu{\tilde{\mu}_{\omega}}
\def\hmu{\hat{\mu}_{\omega}}
\def\A{A_{\mu,\nu}}
\def\mv{\mu_{\nu}^{\omega}}
\def\vv{\hat{\mu}_{\nu}}
\def\M{\mathcal{M}}
\def\HP{\mathcal{H}^{p}}
\def\dsup{\displaystyle\sup}

\begin{document}
\title[  ]{Embedding theorem for weighted Hardy spaces into Lebesgue spaces}
\author{Zengjian Lou \ \ \  Conghui Shen$^\ast$  }

\address{Zengjian Lou\\ Department of Mathematics, Shantou University \\
 Shantou Guangdong 515063, People's Republic of China.}
\email{zjlou@stu.edu.cn }

\address{Conghui Shen\\ Department of Mathematics, Shantou University \\
 Shantou Guangdong 515063, People's Republic of China. }\email{shenconghui2008@163.com }

\begin{abstract} In this paper, we consider the weighted Hardy space $\H^p(\omega)$ induced by an $A_1$ weight $\omega.$
We characterize the positive Borel measure $\mu$  such that the identical operator maps $\H^p(\omega)$  into $L^q(d\mu)$ boundedly when $0<p, q<\infty.$
As an application, we obtain necessary and sufficient conditions for the boundedness of generalized area operators $\A$ from $\H^p(\omega)$ to $L^q(\omega).$
\thanks{$\ast$ Corresponding author.}
\thanks{The research was supported by the NNSF of China (Nos.11571217, 11720101003, 11871293)  and  Key Projects of Fundamental Research in Universities of Guangdong Province (No.2018KZDXM034).}
\end{abstract}
 \maketitle

\section{Introduction}
As usual, we denote the unit disk and the unit circle by $\DD$ and $\PD,$ respectively.
For an arc $I\subseteq\PD,$ denote the normalized Lebesgue measure of $I$ by $|I|=\frac{1}{2\pi}\int_{I}|d\xi|.$
Let $\mathcal{H}(\DD)$ be the set of all analytic functions on $\DD$. For $0<p<\infty,$ the classical Hardy space
$\HP$ consists of all analytic functions $f\in\mathcal{H}(\DD)$ satisfying
\[\|f\|^p_{\HP}:=\dsup_{0<r<1}\frac{1}{2\pi}\int_{\PD}|f(r\xi)|^p|d\xi|<\infty.\]
For $p=\infty,$ we say that $f\in \mathcal{H}^{\infty}$ if $f$ is a bounded analytic function. More about the theory of
Hardy spaces, we refer the readers to \cite{{Du},{Ga},{Zhu}}.
Let $\omega$ be a non-negative function on $\PD$ and $1\leq p<\infty,$ we say that $\omega$ satisfies the $A_p$
condition of Muckenhoupt, denoted by $\omega\in A_P,$ if there exists a positive constant $C$ such that for any arcs $I\subseteq\PD,$
\begin{align*}
\lf(\frac{1}{|I|}\int_{I}\omega(\xi)|d\xi|\rt)
\lf(\frac{1}{|I|}\int_{I}\omega(\xi)^{-1/(p-1)}|d\xi|\rt)^{p-1}\leq C
\end{align*}
or
\begin{align*}
\frac{1}{|I|}\int_{I}\omega(\xi)|d\xi|\leq C \inf_{\xi\in I}\omega(\xi),
\end{align*}
whenever $1<p<\infty$ or $p=1,$ respectively.
As we known,  the Muckenhoupt $A_p$ weight $\omega$ satisfies the doubling property \cite{Gr,To},
that is, there exists a positive constant $C$ such that for all arcs $I\subseteq\PD,$
\begin{align}\label{eq2.1}
\int_{2I}\omega(\xi)|d\xi|\leq C\int_{I}\omega(\xi)|d\xi|.
\end{align}
For $\omega\in A_1,$  $0<p<\infty,$ the weighted Lebesgue space $L^{p}(\omega)$ consists of all complex-valued Lebesgue measurable functions
$f$ on $\PD$ for which
\[\|f\|_{L^{p}(\omega)}:=\lf(\frac{1}{2\pi}\int_{\PD}|f(\xi)|^{p}\omega(\xi)|d\xi|\rt)^{\frac{1}{p}}<\infty.\]
It is well known that for each function $f\in\mathcal{H}^{1},$ the non-tangential limit of $f$ exists almost everywhere, in other words, $f(e^{i\theta})$ is well defined for any $\theta$ except that a zero measure set.
Furthermore, $f(e^{i\theta})\in L^{1}(\PD),$ see {\cite[Theorem 2.2]{Du}}. So we may define the weighted Hardy space
$\mathcal{H}^{P}(\omega)$ as follows:
\[\mathcal{H}^{P}(\omega):=\lf\{f\in\mathcal{H}^{1}:f\in L^{p}(\omega)\rt\}\]
equipped with the norm $\|f\|_{\H^{p}(\omega)}=\|f\|_{L^{p}(\omega)}.$
When $\omega\equiv1,$ $\H^p(\omega)$ coincides with the classical Hardy space $\H^p.$

In 1962, Carleson characterized a positive Borel measure $\mu$ on $\DD$ such that $\H^{P}\subseteq L^{p}({d\mu})$ continuously.
Indeed, he proved that if $1<p<\infty,$
then there exists a positive constant $C$ such that for each $f\in \H^p,$
\[\int_{\DD}|f(z)|^pd\mu(z)\leq C\|f\|^p_{\H^p}\]
holds if and only if $\mu$ is a Carleson measure, see \cite{Ca}.
Duren generalized Carleson's result in \cite{Dur} and proved that
if $0<p<q<\infty,$  then there is a positive constant $C$ such that for all $f\in\H^p,$
\[\int_{\DD}|f(z)|^qd\mu(z)\leq C\|f\|^q_{\H^p}\]
holds if and only if $\mu$ is a $\frac{q}{p}$-Carleson measure.

In \cite{Gi}, Girela, Lorente and Sarri\'{o}n obtained some necessary conditions and some sufficient conditions for positive Borel measure $\mu$ so that the differentiation operator from $\H^p(\omega)$  to $L^p(d\mu)$ is bounded.
In \cite{Lu}, Luecking characterized the positive Borel measures $\mu$ such that differential operator of order $m$ maps $\H^p$ into $L^q(d\mu)$ boundedly.
It is worth mentioning that Mcphail obtained a description of those positive Borel measures $\mu$ on $\DD$ so that the identical mapping
from $\H^p(\omega)$ to $L^p(d\mu)$ is bounded in \cite{Mc}, which also extended Carleson's theorem.

Motivated by above results, we give a characterization of the positive Borel measure $\mu$ such that
the identical operator maps $\H^p(\omega)$ into $L^q(d\mu)$ boundedly when $0<p, q<\infty$ and $\omega\in A_1.$

\begin{theorem}\label{a}
Let $\mu$ be a positive Borel measure on $\DD$ and $\omega$ be a non-negative measurable function on $\PD.$
If $\omega\in A_1$ and $0<p\leq q<\infty,$ then there exists a positive constant $C$ such that for all $f\in\H^p(\omega),$
\begin{align}\label{eq2.2}
\|f\|_{L^q(d\mu)}\leq C\|f\|_{\H^p(\omega)}
\end{align}
holds if and only if there exists a positive constant $C$ such that for all $I\subseteq\PD,$
\begin{align}\label{eq2.3}
\mu(S(I))\leq C\lf(\int_{I}\omega(\xi)|d\xi|\rt)^{\frac{q}{p}}
\end{align}
holds.
\end{theorem}
By the proof of Theorem \ref{a}, we can obtain the following corollary easily.
\begin{corollary}\label{coro1}
Let $\mu$ be a positive Borel measure on $\DD$ and $\omega$ be a non-negative measurable function on $\PD.$
If $\omega\in A_1$ and $0<p\leq q<\infty,$  then
\[\M_{\omega}: L^P(\omega)\rightarrow L^q(\mu)\]
is bounded if and only if $\M_{\omega,\frac{q}{p}}(\mu)\in L^{\infty},$ where $\M_{\omega}$ is defined in Section 2.
\end{corollary}

Define the positive Borel measures $\tilde{\mu}_{\omega}$ and $\hat{\mu}_{\omega}$ by
\[\tilde{\mu}_{\omega}(\xi):=\frac{1}{2\pi\omega(\xi)}\int_{\DD}\frac{1-|z|^2}{|\xi-z|^2}d\mu(z)\]
and
\[\hat{\mu}_{\omega}(\xi):=\frac{1}{\omega(\xi)}\int_{\Gamma(\xi)}\frac{1}{1-|z|}d\mu(z),\]
where  $\Gamma(\xi)$ is defined by
\[\Gamma(\xi):=\lf\{z\in\DD:|\xi-z|<2(1-|z|)\rt\}.\]

\begin{theorem}\label{b}
Let $\mu$ be a positive Borel measure on $\DD$ and $\omega$ be a non-negative measurable function on $\PD.$
If $\omega\in A_1$ and $0<q<p<\infty,$ then the following conditions are equivalent:

{\rm (i)} $\|f\|_{L^q(d\mu)}\leq C\|f\|_{\H^p(\omega)}$ for some positive constant $C.$

{\rm (ii)} ~~$\tmu\in L^{\frac{p}{p-q}}(\omega).$

{\rm (iii)}~~$\hmu\in L^{\frac{p}{p-q}}(\omega).$
\end{theorem}

Let $\mu$ be a positive Borel measure on $\DD,$ the area operator $A_{\mu}$ is defined by
\[A_{\mu}f(\xi):=\int_{\Gamma(\xi)}|f(z)|\frac{1}{1-|z|}d\mu(z),~~~\xi\in\PD,~~~f\in\H(\DD).\]
If $\nu$ is a positive Borel measure on $\PD,$ we denote $d\mu_{\nu}^{\omega}(z):=\frac{\omega(T(z))}{\nu(T(z))}d\mu(z)$
and define the generalized area operator $A_{\mu,\nu}$ by:
\[A_{\mu,\nu}f(\xi):=\int_{\Gamma(\xi)}|f(z)|\frac{1}{\nu(T(z))}d\mu(z),~~~\xi\in\PD,~~~f\in\H(\DD).\]
Here $T(z)$ is defined by
\[T(z):=\lf\{\xi\in\PD: z\in\Gamma(\xi)\rt\}.\]
It is well known that
\begin{align}\label{e1}
\int_{\PD}\chi_{\Gamma(\xi)}(z)|d\xi|\asymp 1-|z|, ~~~\forall z\in\DD
\end{align}
and
\begin{align*}
 |T(z)|\asymp 1-|z|, \forall z\in \Gamma(\xi).
\end{align*}
Area operators is an important research topic in harmonic analysis. It relates to, for instance, Poisson integral, the non-tangential maximal function,
Littlewood-Paley operators and tent spaces, etc. Cohn studied the area operators $A_{\mu}$ from the Hardy spaces $\H^p$ to $L^p$ in \cite{Co},
it states that for a positive Borel measure $\mu$ on $\DD$ and $0<p<\infty,$ $A_{\mu}$ from $\H^p$ to $L^p$ is bounded if and only if $\mu$ is a Carleson measure. The approach by Cohn relies on  John and Nirenberg's estimate and Calder\'{o}n-Zygmund decomposition.
More recently, Gong,  Wu  and the first author extended Cohn's result in \cite{Go}. They used the same ideas and together with the Riesz factorization
of Hardy spaces. Wu characterized the positive Borel measure $\mu$ on the unit disk $\DD$ for which the area operator is bounded from standard weighted Bergman
space $A_{\alpha}^p$ to Lebesgue space $L^p(\PD)$ in \cite{Wu, WuZ}. In \cite{Pel}, Pel\'{a}ez, R\"{a}tty\"{a} and Sierra gave the
boundedness and compactness of the generalized area operator from weighted Bergman spaces to weighted Lebesgue spaces.

As an application of Theorems \ref{a} and \ref{b},
we obtained necessary and sufficient conditions for the boundedness of $\A$ from $\H^p(\omega)$ to $L^q(\omega)$.

\begin{theorem}\label{c}
Let $\mu$ and $\nu$ be positive Borel measures on $\DD$ and $\PD$ respectively.
Suppose that $\omega$ is a non-negative measurable function on $\PD$ satisfies
$\mu(\{z\in\DD:\nu(T(z))=0\})=0=\mu(\{z\in\DD:\omega(T(z))=0\}).$
If $\omega\in A_1$ and $0<p\leq q<\infty,$ then $\A: \H^p(\omega)\rightarrow L^q(\omega)$ is bounded
if and only if there exists a positive constant $C$ such that for all arcs $I\subseteq \PD,$
\begin{align}\label{e2}
\mv(S(I))\leq C\lf(\int_{I}\omega(\xi)|d\xi|\rt)^{1+\frac{1}{p}-\frac{1}{q}}
\end{align}
holds.
\end{theorem}
\begin{theorem}\label{d}
Let $\mu$ and $\nu$ be positive Borel measures on $\DD$ and $\PD$ respectively.
Suppose that $\omega$ is a non-negative measurable function on $\PD$ satisfies
$\mu(\{z\in\DD:\nu(T(z))=0\})=0=\mu(\{z\in\DD:\omega(T(z))=0\}).$
If $\omega\in A_1$ and $1\leq q<p<\infty,$ then $\A: \H^p(\omega)\rightarrow L^q(\omega)$ is bounded if and only if
the following function
\[\hat{\mu}_{\nu}(\xi):=\int_{\Gamma(\xi)}\frac{1}{\nu(T(z))}d\mu(z)\]
belongs to $L^{\frac{pq}{p-q}}(\omega).$
\end{theorem}

Throughout this paper, $C$ is a positive constant depending only on index $p,q,\alpha,\cdots,$ not necessary to be the same from one line to another.
Let $f$ and $g$ be two positive functions, for convenience, we write $f\preceq g,$  if $f\leq Cg$ holds. If $f\preceq g$ and $g\preceq f,$ then we set $f\asymp g.$ Also, we write $\omega(I)=\int_{I}\omega(\xi)|d\xi|.$

\section{Preliminary results}
In this section, we state some definitions and results will be used in the paper.
For an arc $I\subseteq \PD,$ the Carleson square based on $I$ is defined by
\[S(I):=\lf\{z\in\DD:1-|I|\leq |z|<1,\frac{z}{|z|}\in I\rt\}.\]
If $I=\PD,$ then $S(I)=\DD.$
Let $\mu$ be a positive Borel measure on $\DD,$ for $0<s<\infty,$ $\mu$ is called an $s$-Carleson measure if there is a positive constant $C$ such that for all arcs $I\subseteq\PD,$
 \[ \mu(S(I))\leq C|I|^{s}.\]
1-Carleson measures are the classical Carleson measures.

Given $f\in L^1(\omega)$ and $z=re^{i\theta}\in\DD,$ we associate the boundary arc
\[I_z:=\lf\{e^{it}:\theta-\frac{1-r}{2}\leq t\leq \theta+\frac{1-r}{2}\rt\}\]
and define the weighted maximal function
\[\M_{\omega}f(z):=\dsup_{I}\frac{1}{\omega(I)}\int_{I}|f(\xi)|\omega(\xi)|d\xi|,\]
where the supremum is taken over all arcs $I\supseteq I_z.$
When $\alpha>0$ and $\mu$ is a positive Borel measure on $\DD,$ denote
\[\M_{\omega,\alpha}(\mu)(z):=\sup_{I\supseteq Iz}\frac{\mu(S(I))}{(\omega(I))^{\alpha}}.\]
We also set $\M_{\omega}(\mu):=\M_{\omega,1}(\mu).$

Suppose that $f$ is a harmonic function on $\DD,$ we define the non-tangential maximal function $Nf$ as the following:
\[Nf(\xi):=\sup_{z\in\Gamma(\xi)}|f(z)|.\]
It is known that $N$ is bounded on $L^p(\omega)$ when $1<p<\infty$ and $\omega\in A_1.$

For $f\in L^{1}(\PD),$  the Poisson integral of $f$ is defined by
\[Pf(z):=\frac{1}{2\pi}\int_{\PD}\frac{1-|z|^{2}}{|\xi-z|^2}f(\xi)|d\xi|.\]
Then $Pf$ is the harmonic extension of $f$ onto $\DD.$ Let
\[Qf(z):=\frac{1}{2\pi}\int_{\PD}\frac{\xi\bar{z}-\bar{\xi}z}{|\xi-z|^2}f(\xi)|d\xi|,\]
we know that $Qf$ is a conjugate to $Pf,$ that is,
\[\H f(z):=Pf(z)+iQf(z)\]
is analytic on $\DD.$ It is standard that for $1<p<\infty$ and $\omega\in A_1,$  an $f$ belongs to $L^p(\omega)$
if and only if $\H(f)$ is in $\H^p(\omega).$ In fact, the operator $\H$ defined above is bounded from $L^p(\omega)$ into $\H^p(\omega).$

In the proofs of theorems, we need the following lemmas.
\begin{lemma}\label{lmm5}
Let $\omega\in A_1$ and $0<t,s,r<\infty$ such that $\frac{1}{t}=\frac{1}{s}+\frac{1}{r}.$ Then $\H^r(\omega)\cdot\H^s(\omega)=\H^t(\omega).$
\end{lemma}
\begin{proof}
For $f\in\H^s(\omega),$ $g\in\H^r(\omega),$ H\"{o}lder's inequality yields that
\begin{align*}
\int_{\PD}|f(\xi)g(\xi)|^t\omega(\xi)|d\xi|
&\leq\lf(\int_{\PD}(|f(\xi)|^t)^{\frac{s}{t}}\omega(\xi)|d\xi|\rt)^{\frac{t}{s}}\lf(\int_{\PD}(|g(\xi)|^t)^{\frac{r}{t}}\omega(\xi)|d\xi|\rt)^{\frac{t}{r}}\\
&=\lf(\int_{\PD}|f(\xi)|^s\omega(\xi)|d\xi|\rt)^{\frac{t}{s}}\lf(\int_{\PD}|g(\xi)|^r\omega(\xi)|d\xi|\rt)^{\frac{t}{r}}.
\end{align*}
Therefore, $fg\in\H^t(\omega)$ and $\|fg\|_{\H^t(\omega)}\leq \|f\|_{\H^s(\omega)}\|g\|_{\H^r(\omega)}.$

On the other hand, if $f\in\H^t(\omega)\subseteq\H^1,$ then there exists a Blaschke product $B$ and an $\H^1$ function $G$
such that $f=BG.$ Moreover, $G$ does not vanish in $\DD.$ We have $|B(\xi)|=1$ almost everywhere, so $|f(\xi)|=|G(\xi)|$ almost everywhere.
Hence, $f\in\H^t(\omega)$ implies that $G\in\H^t(\omega).$ Since $G$ does not vanish in $\DD,$ $G$ can be factored in the form $G=G^{t/s}G^{t/r}.$
Let $g=BG^{t/s},$ $h=G^{t/r},$ then $f=gh.$
Since
\begin{align*}
\int_{\PD}|g(\xi)|^s\omega(\xi)|d\xi|=\int_{\PD}|G(\xi)|^t\omega(\xi)|d\xi|<\infty
\end{align*}
and
\begin{align*}
\int_{\PD}|h(\xi)|^r\omega(\xi)|d\xi|=\int_{\PD}|G(\xi)|^t\omega(\xi)|d\xi|<\infty,
\end{align*}
these imply that $g\in\H^s(\omega)$ and $h\in\H^r(\omega).$
\end{proof}

\begin{lemma}\label{lmm 4}
Let $\omega\in A_1.$ Then there exists a positive constant $C$ such that
\begin{align*}
|f(z)|\leq C\M_{\omega}f(z), \  \  \  \   z\in \DD,
\end{align*}
for all $f\in \H(\DD).$
\end{lemma}

The proof of Lemma \ref{lmm 4} is similar to that of Lemma 2.5 in \cite{Pela}, the details is omitted here.

\section{Proofs of main theorems}
\begin{proof}[Proof of Theorem \ref{a}]
Assume that (\ref{eq2.2}) holds first.
Let $I\subseteq\PD$ be an open arc, we may assume that $|I|<\frac{1}{2}$ (the argument of the case $|I|\geq\frac{1}{2}$ is similar and is omitted here).
We can find a point $a=a(I)\in\DD$ such that $\frac{a}{|a|}$ is the center of $I$ and $1-|a|=|I|,$ that is $I=I_a.$
Define the arcs $I_k$ by
\[I_k:=2^kI, ~~~~k=0,1\cdots,~m,~~~~I_{m+1}=\PD,\]
where $m$ is the largest natural number satisfies $2^m|I|<1.$
Since $\omega\in A_1,$ we fix a positive constant $M$ such that (\ref{eq2.1}) holds.
Let
\[g(z)=\frac{1}{1-\bar{a}z},~~~|z|\leq 1.\]
By the same argument as in \cite{Gi} we see that if $a=|a|e^{i\varphi},$ then
\begin{align}\label{eq2.4}
|g(e^{i\theta})|\leq \frac{3}{(|I|^2+|\theta-\varphi|^2)^{\frac{1}{2}}}.
\end{align}
It is easy to see that
\begin{align}\label{eq2.5}
|g(e^{i\theta})|\leq \frac{3}{|I|},~~~if ~e^{i\theta}\in I_0=I,
\end{align}
and
\begin{align}\label{eq2.6}
|g(e^{i\theta})|\leq \frac{3}{2^{k-1}|I|},~~~if ~e^{i\theta}\in I_{k+1}\setminus I_k,~~~k=0,1\cdots,~m.
\end{align}
Choosing an integer $N$ sufficient large such that $\frac{M}{2^{Np}}<1.$
Setting $f(z)=g(z)^N,$ then $f\in\H^p(\omega).$ Combining with (\ref{eq2.1}), (\ref{eq2.5}) and (\ref{eq2.6}) and using the definition of $I_k,$ we deduce
\begin{align*}
\|f\|^p_{\H^p(\omega)}
&=\frac{1}{2\pi}\int_{I_0}|g(\xi)|^{Np}\omega(\xi)|d\xi|+\sum_{k=0}^m\frac{1}{2\pi}\int_{I_{k+1}\setminus I_k}|g(\xi)|^{Np}\omega(\xi)|d\xi|\\
&\leq \frac{3^{Np}}{2\pi|I|^{Np}}\int_{I}\omega(\xi)|d\xi|+
\frac{1}{2\pi}\sum_{k=0}^m\frac{3^{Np}}{2^{(k-1)Np}|I|^{Np}}\int_{I_{k+1}\setminus I_k}\omega(\xi)|d\xi|\\
&\leq \frac{3^{Np}}{2\pi|I|^{Np}}\lf(\int_{I}\omega(\xi)|d\xi|+\sum_{k=0}^m\frac{M^k}{2^{(k-1)Np}}\int_{I}\omega(\xi)|d\xi|\rt)\\
&\leq \frac{3^{Np}}{2\pi|I|^{Np}}\lf(1+M^2\sum_{k=0}^{\infty}\lf(\frac{M}{2^{Np}}\rt)^{k-1}\rt)\int_{I}\omega(\xi)|d\xi|.
\end{align*}
Since $\frac{M}{2^{Np}}<1,$ we have
\begin{align}\label{eq2.7}
\|f\|^p_{\H^p(\omega)}\leq \frac{C}{|I|^{Np}}\int_{I}\omega(\xi)|d\xi|.
\end{align}
For $a=|a|e^{i\varphi}$ and $z\in S(I),$ a simple geometric argument shows that
\begin{align*}
\lf|\frac{1}{|a|}e^{i\varphi}-z\rt|\leq \lf|\frac{1}{|a|}e^{i\varphi}-|a|e^{i(\varphi+\frac{1}{2}|I|)}\rt|
=\frac{1}{|a|}\lf(1-2|a|^2\cos\lf(\frac{|I|}{2}\rt)+|a|^4\rt)^{\frac{1}{2}}.
\end{align*}
Thus,
\begin{align}\label{eq2.8}
\frac{1}{|g(z)|}
\leq\lf(1-2|a|^2\cos\lf(\frac{|I|}{2}\rt)+|a|^4\rt)^{\frac{1}{2}}\leq 3|I|.
\end{align}
Therefore, applying (\ref{eq2.2}), (\ref{eq2.7}) and (\ref{eq2.8}), we obtain
\begin{align*}
\mu(S(I))&\leq \int_{S(I)}\bigg(|g(z)|\cdot3|I|\bigg)^{Nq}d\mu(z)\\
&\leq C|I|^{Nq}\|f\|^q_{L^q(d\mu)}\\
&\leq C|I|^{Nq}\|f\|^q_{\H^P(\omega)}\\
&\leq C\lf(\int_{I}\omega(\xi)|d\xi|\rt)^{\frac{q}{p}}.
\end{align*}
This proves (\ref{eq2.3}).

Conversely, for each $s>0,$ we defined the sets $E_s$ as follows:
\[E_s=\{z\in\DD: \M_{\omega}f(z)>s\}.\]
We first show that there exists a positive constant $C\geq 1$ such that for all $f\in\H^1(\omega)$ and $s>0,$
\begin{align}\label{eq2.9}
\mu(E_s)\leq Cs^{-\frac{q}{p}}\|f\|^{\frac{q}{p}}_{\H^1(\omega)}.
\end{align}
If $E_s=\emptyset,$ then (\ref{eq2.9}) is clearly true.
Now suppose that $E_s$ is nonempty. Given $\varepsilon>0,$ write
\[A^{\varepsilon}_{s}=\lf\{z\in\DD: \int_{I_z}|f(\xi)|\omega(\xi)|d\xi|>s(\varepsilon+\omega(I_z))\rt\}\]
and
\[B^{\varepsilon}_{s}=\lf\{z\in\DD: I_z\subseteq I_w~~ for ~~some~~ w\in A^{\varepsilon}_{s}\rt\}.\]
Clearly, $E_s=\displaystyle\bigcup_{\varepsilon>0}B^{\varepsilon}_{s}.$
It is easy to see that $B^{\varepsilon}_{s}$ become larger if $\varepsilon\downarrow 0^+$ and that
\begin{align}\label{eq2.10}
\mu(E_s)=\lim_{\varepsilon\rightarrow 0+}\mu(B^{\varepsilon}_{s}).
\end{align}
Notice that for every $\varepsilon>0$ and $s>0,$ there exists finitely many points $z_n\in A^{\varepsilon}_{s}$ such that
the arcs $I_{z_n}$ are disjoint. In fact, if this is not the case, then we can find a sequence $\{z_n\}_{n=1}^{\infty}\subseteq A^{\varepsilon}_{s}$
such that $I_{z_n}\bigcap I_{z_m}=\emptyset,$ $\forall n\neq m.$ Therefore,
\begin{align}\label{eq2.11}
s\sum_{n}(\varepsilon+\omega(I_z))\leq \sum_{n}\int_{I_z}|f(\xi)|\omega(\xi)|d\xi|\leq \|f\|_{\H^1(\omega)}.
\end{align}
So $\|f\|_{\H^1(\omega)}\geq s\displaystyle\sum_{n}\varepsilon=\infty,$ this is a contradiction.
By Covering Lemma, there exist finite many points $z_1, z_2, \cdots, z_m\in A^{\varepsilon}_{s}$ satisfying the following conditions:

(I) The arcs $I_{z_j}$ are disjoint,

(II) $A^{\varepsilon}_{s}\subseteq \displaystyle\bigcup_{j=1}^{m}\lf\{z\in\DD: I_z\subseteq J_{z_j}\rt\},$ where $J_z$ is the arc of length 5$|I_z|$ whose center
coincides with that of $I_z.$

It follows that
\begin{align}\label{eq2.12}
B^{\varepsilon}_{s}\subseteq \displaystyle\bigcup_{j=1}^{m}\lf\{z\in\DD: I_z\subseteq J_{z_j}\rt\}.
\end{align}
Observe that (\ref{eq2.3}) and the assumption of $\omega$ yield
\begin{align*}
\mu\lf(\lf\{z\in\DD: I_z\subseteq J_{z_j}\rt\}\rt)
&\leq \mu(S(J_{z_j}))\\
&\preceq \omega(J_{z_j})^{\frac{q}{p}}\\
&\preceq \omega(I_{z_j})^{\frac{q}{p}}, ~~~j=1, 2, \cdots, m.
\end{align*}
Combining (\ref{eq2.11}) and (\ref{eq2.12}), we get
\begin{align}\label{eq2.13}
\mu(B^{\varepsilon}_{s})\preceq\sum_{j=1}^m\omega(I_{z_j})^{\frac{q}{p}}
\leq\lf(\sum_{j=1}^m\omega(I_{z_j})\rt)^{\frac{q}{p}}
\leq\lf(\frac{1}{s}\lf\|f\rt\|_{\H^1(\omega)}\rt)^{\frac{q}{p}}.
\end{align}
Using (\ref{eq2.10}) and (\ref{eq2.13}), we see that (\ref{eq2.9}) is true with a constant $C\geq 1$ depending only on $p, q$ and $\omega.$
Fix $\alpha>\frac{1}{p},$ for $s>0,$ let $|f|^{\frac{1}{\alpha}}=g_{\frac{1}{\alpha}, s}+\chi_{\frac{1}{\alpha},s},$
where

$$g_{\frac{1}{\alpha}, s}(\xi)=\left\{
          \begin{array}{cc}
           ~~ \lf|f(\xi)\rt|^{\frac{1}{\alpha}}, & \hbox{~~~~~~~~if ~$\lf|f(\xi)\rt|^{\frac{1}{\alpha}}>\frac{s}{2C}$ ;} \\
          ~~  0, & \hbox{otherwise.}
          \end{array}
        \right.$$
Note that $p>\frac{1}{\alpha},$ the function $g_{\frac{1}{\alpha}, s}$ belongs to $L^1_{\omega}(\PD)$ for all $s>0.$ Moreover,
\begin{align*}
\M_{\omega}(|f|^{\frac{1}{\alpha}})\leq \M_{\omega}(g_{\frac{1}{\alpha}, s})+\M_{\omega}(\chi_{\frac{1}{\alpha}, s})
\leq\M_{\omega}(g_{\frac{1}{\alpha}, s})+\frac{s}{2C}.
\end{align*}
So
\begin{align}\label{eq2.14}
\lf\{z\in\DD: \M_{\omega}(|f|^{\frac{1}{\alpha}})>s\rt\}\subseteq \lf\{z\in\DD: \M_{\omega}(g_{\frac{1}{\alpha}, s})>\frac{s}{2}\rt\}.
\end{align}
 By Lemma \ref{lmm 4} and Minkowski's inequality of continuous form, we deduce that
\begin{align*}
\int_{\DD}|f(z)|^qd\mu(z)&\preceq\int_{\DD}(\M_{\omega}(|f|^{\frac{1}{\alpha}})(z))^{q\alpha}d\mu(z)\\
&\leq q\alpha\int_0^{\infty}s^{q\alpha-1}\mu\lf(\lf\{z\in\DD: \M_{\omega}(g_{\frac{1}{\alpha}, s})>\frac{s}{2}\rt\}\rt)ds\\
&\preceq\int_0^{\infty}s^{q\alpha-1-\frac{q}{p}}\|g_{\frac{1}{\alpha}, s}\|^{\frac{q}{p}}_{L^1_{\omega}}ds\\
&\preceq \lf(\int_{\PD}|f(\xi)|^{\frac{1}{\alpha}}\omega(\xi)
\lf(\int_{0}^{2C|f(\xi)|^{\frac{1}{\alpha}}}s^{q\alpha-1-\frac{q}{p}}ds\rt)^{\frac{p}{q}}|d\xi|\rt)^{\frac{q}{p}}\\
&\preceq \|f\|^q_{\H^p(\omega)}.
\end{align*}
This proves (\ref{eq2.2}), we complete the proof of Theorem \ref{a}.
\end{proof}

\begin{proof}[Proof of Theorem \ref{b}]
Let us recall the Riesz factorization theorem for $\H^p$ spaces first. For each $f\in\H^p$ and $f\neq0,$ then $f(z)=F(z)B(z)$ with
$F\in\H^p$ has no zero in $\DD$ and $\|F\|_{\H^p}=\|f\|_{\H^p},$ $B$ is a Blaschke product. So the statement (i) is equivalent to
\begin{align}\label{eq2.15}
\int_{\DD}|f(z)|d\mu(z)\leq C\|f\|_{\H^{\frac{p}{q}}(\omega)}, ~~~\forall f\in\H^{\frac{p}{q}}(\omega).
\end{align}
(i) $\Rightarrow$ (ii). For $h\geq 0$ and $h\in L^{\frac{p}{q}}(\omega),$ by Fubini's theorem,
\begin{align*}
\int_{\DD}|Ph(z)|d\mu(z)&=\int_{\PD}\frac{1}{2\pi\omega(\xi)}\int_{\DD}\frac{1-|z|^2}{|\xi-z|^2}d\mu(z)h(\xi)\omega(\xi)|d\xi|\\
&=\int_{\PD}\tmu(\xi)h(\xi)\omega(\xi)|d\xi|.
\end{align*}
Since $\frac{p}{q}>1,$ the operator $\H$ is bounded from $L^{\frac{p}{q}}(\omega)$ to $\H^{\frac{p}{q}}(\omega).$
Then (\ref{eq2.15}) implies that
\begin{align*}
\int_{\DD}|Ph(z)|d\mu(z)\leq \int_{\DD}|\H h(z)|d\mu(z)\leq C\|\H h\|_{\H^{\frac{p}{q}}(\omega)}\leq C\|h\|_{L^{\frac{p}{q}}(\omega)}.
\end{align*}
Hence
\[\int_{\PD}\tmu(\xi)h(\xi)\omega(\xi)|d\xi|\leq C\|h\|_{L^{\frac{p}{q}}(\omega)},~~~~\forall h\in L^{\frac{p}{q}}(\omega),~~~h\geq 0.\]
A duality argument shows that $\tmu\in L^{\frac{p}{p-q}}(\omega).$

(ii) $\Rightarrow$ (iii).
If $z\in\Gamma(\xi),$ then $|\xi-z|\leq 2(1-|z|).$
We obtain $\frac{1}{1-|z|}\leq \frac{4(1-|z|^2)}{|\xi-z|^2}$ and
\begin{align*}
\hat{\mu}_{\omega}(\xi)
\preceq \frac{1}{\omega(\xi)}\int_{\Gamma(\xi)}\frac{1-|z|^2}{|\xi-z|^2}d\mu(z)
=\tmu(\xi).
\end{align*}

(iii) $\Rightarrow$ (i). Notice that $\frac{p}{q}>1,$
$NP$ is bounded on $L^{\frac{p}{q}}(\omega).$ For $h\in L^{\frac{p}{q}}(\omega)$ and $h\geq 0,$
Fubini's theorem and H\"{o}lder's inequality yield
\begin{align*}
\int_{\PD}Ph(z)d\mu(z)&\asymp\int_{\PD}\int_{\Gamma(\xi)}Ph(z)\frac{1}{1-|z|}d\mu(z)|d\xi|\\
&\leq\int_{\PD}NP(h)(\xi)\int_{\Gamma(\xi)}\frac{1}{1-|z|}d\mu(z)|d\xi|\\
&\leq\|NP(h)\|_{L^{\frac{p}{q}}(\omega)}\|\hmu\|_{L^{\frac{p}{p-q}}(\omega)}\\
&\leq C\|h\|_{L^{\frac{p}{q}}(\omega)}\|\hmu\|_{L^{\frac{p}{p-q}}(\omega)}.\\
\end{align*}
This proves (\ref{eq2.15}).
\end{proof}

\begin{proof}[Proof of Theorem \ref{c}]
First suppose that (\ref{e2}) holds. For the case $0<q\leq 1,$ H\"{o}lder's inequality yields
\begin{align*}
\|\A f\|^q_{L^q(\omega)}
&\preceq \int_{\PD}(Nf(\xi))^{(1-q)p}\lf(\int_{\Gamma(\xi)}|f(z)|^{1-(1-q)\frac{p}{q}}\frac{1}{\nu(T(z))}d\mu(z)\rt)^q\omega(\xi)|d\xi|\\
&\preceq\lf(\int_{\PD}(Nf(\xi))^{p}\omega(\xi)|d\xi|\rt)^{1-q} \lf(\int_{\PD}\int_{\Gamma(\xi)}|f(z)|^{1+p-\frac{p}{q}}\frac{1}{\nu(T(z))}d\mu(z)\omega(\xi)|d\xi|\rt)^q.\\
\end{align*}
Using Theorem \ref{a} and Fubini's theorem,  we have
\begin{align*}
&\int_{\PD}\int_{\Gamma(\xi)}|f(z)|^{1+p-\frac{p}{q}}\frac{1}{\nu(T(z))}d\mu(z)\omega(\xi)|d\xi|\\
&=\int_{\DD}|f(z)|^{1+p-\frac{p}{q}}\frac{1}{\omega(T(z))}d\mv(z)\int_{T(z)}\omega(\xi)|d\xi|\\
&\preceq \|f\|^{p(1+\frac{1}{p}-\frac{1}{q})}_{\H^p(\omega)}.
\end{align*}
In addition,
\[\int_{\PD}(Nf(\xi))^{p}\omega(\xi)|d\xi|\preceq \|f\|^p_{\H^p(\omega)}.\]
So
\[\|\A f\|_{L^q(\omega)}\preceq \|f\|_{\H^p(\omega)}.\]
For the case $q>1,$ by duality argument, we only need to show  for all
$g\in L^{q'}(\omega)$ and $g\geq 0,$ the following inequality holds:
\begin{align}\label{e3}
\int_{\PD}\A f(\xi)g(\xi)\omega(\xi)|d\xi|\preceq \|f\|_{\H^p(\omega)}\|g\|_{L^{q'}(\omega)}.
\end{align}
By Fubuni's theorem, we have
\begin{align*}
\int_{\PD}\A f(\xi)g(\xi)\omega(\xi)|d\xi|
&=\int_{\PD}g(\xi)\int_{\Gamma(\xi)}|f(z)|\frac{1}{\omega(T(z))}d\mv(z)\omega(\xi)|d\xi|\\
&=\int_{\DD}|f(z)|\frac{1}{\omega(T(z))}\int_{T(z)}g(\xi)\omega(\xi)|d\xi|d\mv(z)\\
&\preceq\int_{\DD}|f(z)|\M_{\omega}g(z)d\mv(z).
\end{align*}
Since $\mv$ satisfies (\ref{e2}), so
 $\M_{\omega}:L^{q'}(\omega)\rightarrow L^{q'(1+\frac{1}{p}-\frac{1}{q})}(\mv)$
is bounded by Corollary \ref{coro1}.
It follows from H\"{o}lder's inequality, Theorem \ref{a} and the boundedness of $\M_{\omega}$ that
\begin{align*}
\int_{\DD}|f(z)|\M_{\omega}g(z)d\mv(z)
\preceq \|f\|_{\H^p(\omega)}\|g\|_{L^{q'}(\omega)}.
\end{align*}
The (\ref{e3}) is proved.

Conversely, suppose that $\A: \H^p(\omega)\rightarrow L^q(\omega)$ is bounded. Firstly, we deal with the case $q\geq 1$.
Fixed $a=|a|e^{i\varphi}\in\DD,$ let $I=I_a,$ where
\[I_a=\lf\{e^{it}\in\PD:\varphi-\frac{1-|a|}{2}\leq t\leq \varphi+\frac{1-|a|}{2}\rt\}.\]
Then $|I|=|I_a|\asymp 1-|a|.$
Define the sets $I_k$ as in the proof of Theorem \ref{a} and set $g(\xi)=\frac{1}{1-\bar{a}\xi},$ $\xi\in\PD.$
By (\ref{eq2.1}), (\ref{eq2.5}) and (\ref{eq2.6}), with a constant $N>0$ to be determined later, we obtain
\begin{align*}
&\int_{\PD}\frac{1}{|1-\bar{a}\xi|^{Np+1}}\omega(\xi)|d\xi|\\
&\preceq \frac{1}{(1-|a|)^{Np+1}}\int_{I_0}\omega(\xi)|d\xi|
+\sum_{k=0}^{m}\lf(\frac{3}{2^{k-1}|I|}\rt)^{Np+1}\int_{I_{k+1}\setminus I_{k}}\omega(\xi)|d\xi|\\
&\preceq \frac{1}{|I|^{Np+1}}\int_{I}\omega(\xi)|d\xi|
+\frac{3^{Np}}{|I|^{Np+1}}\sum_{k=0}^{m}\lf(\frac{1}{2^{Np+1}}\rt)^{k-1}M^k\int_{I}\omega(\xi)|d\xi|\\
&\leq \frac{1}{|I|^{Np+1}}\int_{I}\omega(\xi)|d\xi|
+\frac{3^{Np}M}{|I|^{Np+1}}\sum_{k=0}^{\infty}\lf(\frac{M}{2^{Np+1}}\rt)^{k-1}\int_{I}\omega(\xi)|d\xi|.\\
\end{align*}
Choosing $N$ large enough so that $\frac{M}{2^{Np+1}}<1,$ it follows that
\begin{align}\label{eq2.16}
\int_{\PD}\frac{1}{|1-\bar{a}\xi|^{Np+1}}\omega(\xi)|d\xi|\preceq \frac{1}{|I|^{Np+1}}\int_{I}\omega(\xi)|d\xi|.
\end{align}
Since $|I|=|I_a|\asymp 1-|a|,$
\begin{align*}
\int_{\PD}\frac{(1-|a|)^{Np}}{|1-\bar{a}\xi|^{Np+1}}\omega(\xi)|d\xi|\preceq \frac{1}{|I_a|}\int_{I_a}\omega(\xi)|d\xi|.
\end{align*}
Combining this with a similar argument in  \cite{Dur} (page 157) shows that
\begin{align*}
\int_{\PD}\frac{(1-|a|)^{Np}}{|1-\bar{a}\xi|^{Np+1}}\omega(\xi)|d\xi|\asymp \frac{1}{|I_a|}\int_{I_a}\omega(\xi)|d\xi|.
\end{align*}
Let $f_a(z)=\frac{(1-|a|)^N}{(1-\bar{a}z)^{N+\frac{1}{p}}},$ then
\[\|f_a\|^p_{\H^p(\omega)}\asymp \frac{1}{1-|a|}\omega(I_a).\]
Denote $\Lambda(I)=\{z\in\DD: \overline{I_z}\subseteq I\}.$ It is standard that $|1-\bar{a}z|\asymp 1-|a|,$ $\forall z\in \Lambda(I).$
We see that
\begin{align*}
&\int_{I_a}\lf(\int_{\Gamma(\xi)\bigcap\Lambda(I_a)}\frac{(1-|a|)^N}{|1-\bar{a}z|^{N+\frac{1}{p}}}\cdot\frac{1}{\nu(T(z))}d\mu(z)\rt)^q\omega(\xi)|d\xi|\\
&\asymp \int_{I_a}\lf(\int_{\Gamma(\xi)\bigcap\Lambda(I_a)}\frac{1}{(1-|a|)^{\frac{1}{p}}}\cdot\frac{1}{\nu(T(z))}d\mu(z)\rt)^q\omega(\xi)|d\xi|\\
&\asymp|I_a|^{-\frac{q}{p}}\int_{I_a}\lf(\int_{\Gamma(\xi)\bigcap\Lambda(I_a)}\frac{1}{\nu(T(z))}d\mu(z)\rt)^q\omega(\xi)|d\xi|.\\
\end{align*}
In addition,
\begin{align*}
&\int_{I_a}\lf(\int_{\Gamma(\xi)\bigcap\Lambda(I_a)}\frac{(1-|a|)^N}{|1-\bar{a}z|^{N+\frac{1}{p}}}\cdot\frac{1}{\nu(T(z))}d\mu(z)\rt)^q\omega(\xi)|d\xi|\\
&\leq \int_{\PD}\lf(\int_{\Gamma(\xi)}|f_a(z)|\frac{1}{\nu(T(z))}d\mu(z)\rt)^q\omega(\xi)|d\xi|\\
&=\|\A f_a\|^q_{L^q(\omega)}\\
&\preceq (\omega(I_a))^{\frac{q}{p}}|I_a|^{-\frac{q}{p}}.
\end{align*}
Hence
\[\frac{1}{\omega(I_a)}\int_{I_a}\lf(\int_{\Gamma(\xi)\bigcap\Lambda(I_a)}\frac{1}{\nu(T(z))}d\mu(z)\rt)^q\omega(\xi)|d\xi|
\preceq (\omega(I_a))^{\frac{q}{p}-1}.\]
Since $q\geq 1,$ by Jensen's inequality,
\[\int_{I_a}\int_{\Gamma(\xi)\bigcap\Lambda(I_a)}\frac{1}{\nu(T(z))}d\mu(z)\omega(\xi)|d\xi|\preceq (\omega(I_a))^{1+\frac{1}{p}-\frac{1}{q}}.\]
Because $a\in\DD$ is arbitrary, we deduce that for all arcs $I\subseteq\PD,$
\[\int_{I}\int_{\Gamma(\xi)\bigcap\Lambda(I)}\frac{1}{\nu(T(z))}d\mu(z)\omega(\xi)|d\xi|\preceq (\omega(I))^{1+\frac{1}{p}-\frac{1}{q}}.\]
Thus, $\mv(\Lambda(I))\preceq (\omega(I))^{1+\frac{1}{p}-\frac{1}{q}}.$
This implies that $\mv$ satisfies (\ref{e2}) because there exist Carleson squares $S(I_1)$ and $S(I_2)$ such that
$S(I_1)\subseteq \Lambda(I)\subseteq S(I_2)$ and $|I_1|\asymp |I|\asymp |I_2|.$

Now we prove the case $0<q<1.$ For an arc $I\subseteq\PD,$ set
\[R(I)=\{z=|z|e^{i\theta}\in\DD:\omega(T(z))\leq 2^l\omega(I),e^{i\theta}\in I\}.\]
By a similar argument, we see that
\[\frac{1}{\omega(I)}\int_{I}\lf(\int_{\Gamma(\xi)\bigcap R(I)}\frac{1}{\omega(T(z))}d\mv(z)\rt)^q\omega(\xi)|d\xi|\preceq(\omega(I))^{\frac{q}{p}-1}.\]
So
\begin{align}\label{eq2.17}
\frac{1}{\omega(I)}\int_{I}
\lf(\int_{\Gamma(\xi)\bigcap R(I)}\frac{1}{(\omega(I))^{\frac{1}{p}-\frac{1}{q}}}\frac{1}{\omega(T(z))}d\mv(z)\rt)^q\omega(\xi)|d\xi|\leq C.
\end{align}
Fixed $l$ large enough so that for each $\xi\in I,$
\[\hat{R}(I)=\{z\in R(I):2^{l-1}\omega(I)\leq\omega(T(z))\leq 2^l\omega(I)\}\]
is a subset of $\Gamma(\xi).$ It is clear that
\[\bigcup_{\xi\in Q\in\mathcal{D}(I)}\hat{R}(Q)\subseteq R(I)\cap\Gamma(\xi),\]
where $\mathcal{D}(I)$ is the set of all dichotomy arcs of $I.$ Therefore (\ref{eq2.17}) implies
\begin{align}\label{eq2.18}
\frac{1}{\omega(I)}\int_{I}\lf(\frac{1}{(\omega(I))^{\frac{1}{p}-\frac{1}{q}}}
\sum_{Q\in\mathcal{D}(I)}\frac{\mv(\hat{R}(Q))}{\omega(Q)}\chi_{Q}(\xi)\rt)^q\omega(\xi)|d\xi|\leq C.
\end{align}
Set
\[F_{I}(\xi)=\frac{1}{C^{\frac{1}{q}}(\omega(I))^{\frac{1}{p}-\frac{1}{q}}}
\sum_{Q\in\mathcal{D}(I)}\frac{\mv(\hat{R}(Q))}{\omega(Q)}\chi_{Q}(\xi),~~~\xi\in I.\]
Then (\ref{eq2.18}) is equivalent to
\begin{align}\label{eq2.19}
\frac{1}{\omega(I)}\int_{I}(F_I(\xi))^q\omega(\xi)|d\xi|\leq 1,~~~\forall I\subseteq\PD.
\end{align}
If we have the John-Nirenberg type estimate
\begin{align}\label{eq2.20}
\omega\lf(\{\xi\in I: F_{I}(\xi)>t\}\rt)\leq Ce^{-\lambda t^q}\omega(I).
\end{align}
From (\ref{eq2.20}), it is not difficult to see that
\begin{align*}
\omega\lf(\lf\{\xi\in I:\sum_{Q\in\mathcal{D}(I)}\frac{\mv(\hat{R}(Q))}{\omega(Q)}\chi_{Q}(\xi)>t\rt\}\rt)
\leq \exp\lf\{-\lambda(\omega(I))^{1-\frac{q}{p}}t^q\rt\}\omega(I).
\end{align*}
Thus,
\begin{align*}
\mv(\Lambda(I))&\leq \mv(R(I))\\
&=\sum_{Q\in\mathcal{D}(I)}\mv(\hat{R}(Q))\\
&=\int_{I}\lf(\sum_{Q\in\mathcal{D}(I)}\frac{\mv(\hat{R}(Q))}{\omega(Q)}\chi_{Q}(\xi)\rt)\omega(\xi)|d\xi|\\
&=\int_0^{\infty}\omega\lf(\lf\{\xi\in I:\sum_{Q\in\mathcal{D}(I)}\frac{\mv(\hat{R}(Q))}{\omega(Q)}\chi_{Q}(\xi)>t\rt\}\rt)dt\\
&\leq C(\omega(I))^{1+\frac{1}{p}-\frac{1}{q}}.
\end{align*}
Hence,
\[\mv(S(I))\preceq \lf(\omega(I)\rt)^{1+\frac{1}{p}-\frac{1}{q}}.\]

In the following, we use the method of Calder\'{o}n-Zygmund decomposition to prove (\ref{eq2.20}).
Fixed $\alpha>1,$ proceed the Calder\'{o}n-Zygmund decomposition to the function $F_I$ at hight $\alpha,$ there exist $I^1_j\in\mathcal{D}(I),$
$j=1,2,\cdots,$ such that

(I) $|F_I(\xi)|^q\leq \alpha, \xi\in I\setminus\bigcup_jI_j^1;$

(II) $\alpha\leq\frac{1}{\omega(I_j^1)}\int_{I_j^1}(F_I(\xi))^q\omega(\xi)|d\xi|\leq 2\alpha,~~~j=1,2,\cdots;$

(III) $\omega(\bigcup_jI_j^1)=\sum_j\omega(I_j^1)\leq \frac{1}{\alpha}\int_{I}(F_I(\xi))^q\omega(\xi)|d\xi|\leq\frac{1}{\alpha}\omega(I).$

Let $E^1=\cup_jI_j^1,$ then
\[\lf\{\xi\in I:(F_I(\xi))^q>\alpha\rt\}\subseteq E^1\]
and
\[\frac{1}{\omega(I_j^1)}\int_{I_j^1}\lf(F_{I_j^1}(\xi)\rt)^q\omega(\xi)|d\xi|\leq 1,~~~j=1,2,\cdots.\]
Proceed the decomposition to the functions $F^q_{I_j^1},$ we get $I^{2,j}_l\in\mathcal{D}(I_j^1),$ $l=1,2,\cdots,$ satisfying

(I) $|F_{I_j^1}(\xi)|^q\leq \alpha, \xi\in I_j^1\setminus\bigcup_lI_l^{2,j};$

(II) $\alpha\leq\frac{1}{\omega(I_l^{2,j})}\int_{I_l^{2,j}}(F_{I_j^1}(\xi))^q\omega(\xi)|d\xi|\leq 2\alpha,~~~l=1,2,\cdots;$

(III) $\omega(\bigcup_lI_l^{2,j})=\sum_j\omega(I_l^{2,j})\leq \frac{1}{\alpha}\int_{I_j^1}(F_{I_j^1}(\xi))^q\omega(\xi)|d\xi|\leq\frac{1}{\alpha}\omega(I_j^1).$

Denote $E_j^2=\bigcup_lI_l^{2,j}.$ Since $0<p\leq q<1,$  $\lf(\frac{\omega(I_j^1)}{\omega(I)}\rt)^{\frac{1}{p}-\frac{1}{q}}\leq 1.$
If $\xi\in I_j^1\setminus E_j^2,$ then
\begin{align*}
F_{I}(\xi)&=\lf(\frac{\omega(I_j^1)}{\omega(I)}\rt)^{\frac{1}{p}-\frac{1}{q}}F_{I_j^1}(\xi)+F_{I}(\xi)-
\lf(\frac{\omega(I_j^1)}{\omega(I)}\rt)^{\frac{1}{p}-\frac{1}{q}}F_{I_j^1}(\xi)\\
&\leq (\alpha)^{\frac{1}{q}}+F_{I}(\xi)-\lf(\frac{\omega(I_j^1)}{\omega(I)}\rt)^{\frac{1}{p}-\frac{1}{q}}F_{I_j^1}(\xi).\\
\end{align*}
Since $Q\in\mathcal{D}(I)\setminus\mathcal{D}(I_j^1),$ we have $I_j^1\subseteq Q$ and $\chi_Q(\xi)=1.$ So
\begin{align*}
&F_{I}(\xi)-\lf(\frac{\omega(I_j^1)}{\omega(I)}\rt)^{\frac{1}{p}-\frac{1}{q}}F_{I_j^1}(\xi)\\
&=\lf(\frac{1}{\omega(I)}\rt)^{\frac{1}{p}-\frac{1}{q}}\frac{1}{C^{\frac{1}{q}}}
\lf(\sum_{Q\in\mathcal{D}(I)}\frac{\mv(\hat{R}(Q))}{\omega(Q)}\chi_{Q}(\xi)-\sum_{S\in\mathcal{D}(I_j^1)}\frac{\mv(\hat{R}(S))}{\omega(S)}\chi_{S}(\xi)\rt)\\
&=\lf(\frac{1}{\omega(I)}\rt)^{\frac{1}{p}-\frac{1}{q}}\frac{1}{C^{\frac{1}{q}}}
\sum_{Q\in\mathcal{D}(I)\setminus\mathcal{D}(I_j^1)}\frac{\mv(\hat{R}(Q))}{\omega(Q)}\chi_{Q}(\xi)\\
&=\lf(\frac{1}{\omega(I)}\rt)^{\frac{1}{p}-\frac{1}{q}}\frac{1}{C^{\frac{1}{q}}}
\sum_{Q\in\mathcal{D}(I)\setminus\mathcal{D}(I_j^1)}\frac{\mv(\hat{R}(Q))}{\omega(Q)}.
\end{align*}
This shows that $F_{I}(\xi)-\lf(\frac{\omega(I_j^1)}{\omega(I)}\rt)^{\frac{1}{p}-\frac{1}{q}}F_{I_j^1}(\xi)$ is a non-negative constant on $I_j^1,$
which implies
\begin{align*}
&F_{I}(\xi)-\lf(\frac{\omega(I_j^1)}{\omega(I)}\rt)^{\frac{1}{p}-\frac{1}{q}}F_{I_j^1}(\xi)\\
&=\lf(\frac{1}{\omega(I_j^1)}\int_{I_j^1}\lf(F_{I}(\eta)-\lf(\frac{\omega(I_j^1)}{\omega(I)}\rt)^{\frac{1}{p}-\frac{1}{q}}F_{I_j^1}(\eta)\rt)^q
\omega(\eta)|d\eta|\rt)^{\frac{1}{q}}\\
&\leq\lf(\frac{1}{\omega(I_j^1)}\int_{I_j^1}\lf(F_{I}(\eta)\rt)^q\omega(\eta)|d\eta|\rt)^{\frac{1}{q}}\\
&\leq (2\alpha)^{\frac{1}{q}}.
\end{align*}
Now, for all $\xi \in I_j^1\setminus E_j^2,$ $j=1,2,\cdots,$ we have
\[F_I(\xi)\leq (\alpha)^{\frac{1}{q}}+F_{I}(\xi)-\lf(\frac{\omega(I_j^1)}{\omega(I)}\rt)^{\frac{1}{p}-\frac{1}{q}}F_{I_j^1}(\xi)
\leq (\alpha)^{\frac{1}{q}}+(2\alpha)^{\frac{1}{q}}\leq (3\alpha)^{\frac{1}{q}}.\]
Let $E^2=\bigcup_j E_j^2.$ Then  $\lf\{\xi\in I: (F_I(\xi))^q>3\alpha\rt\}\subseteq E^2$ and
\[\omega(E^2)=\sum_j(\omega(E_j^2))\leq \sum_j\frac{1}{\alpha}\omega(I_j^1)\leq \frac{1}{\alpha^2}\omega(I).\]
Repeat the process above, we obtain a sequence $\{E^n\}$ such that
\[\omega(E^n)\leq \frac{1}{\alpha^n}\omega(I)\]
and
\begin{align}\label{eq2.21}
\lf\{\xi\in I:(F_I(\xi))^q>3(n-1)\alpha\rt\}\subseteq E^n.
\end{align}
From (\ref{eq2.21}),
\begin{align}\label{eq2.22}
\omega(\{\xi\in I:(F_I(\xi))^q>\alpha\})\leq Ce^{-\lambda \alpha}\omega(I).
\end{align}
This easily implies the John-Nirenberg type estimate (\ref{eq2.20}). This proof is complete.
\end{proof}

\begin{proof}[Proof of Theorem \ref{d}]
Suppose that $\vv\in L^{\frac{pq}{p-q}}(\omega).$ For all $f\in\H^p(\omega),$ we have
\begin{align*}
\A f(\xi)\leq Nf(\xi)\vv(\xi).
\end{align*}
Using H\"{o}lder's inequality implies
\begin{align*}
\|\A f\|^q_{L^q(\omega)}&\leq\int_{\PD}(Nf(\xi))^q(\vv(\xi))^q\omega(\xi)|d\xi|\\
&\leq\lf(\int_{\PD}(Nf(\xi))^p\omega(\xi)|d\xi|\rt)^{\frac{q}{p}}\lf(\int_{\PD}(\vv(\xi))^{\frac{pq}{p-q}}\omega(\xi)|d\xi|\rt)^{\frac{p-q}{p}}\\
&\preceq\|f\|^q_{\H^p(\omega)}\|\vv\|^q_{L^{\frac{pq}{p-q}}(\omega)}.
\end{align*}
Hence, $\A$ is bounded from $\H^p(\omega)$ to $L^q(\omega).$

Conversely, for any $f\in \H^{p}(\omega)$ and any $g\in \H^{q'}(\omega),$  Fubini's theorem yields
\begin{align*}
&\int_{\PD}\lf(\int_{\Gamma(\xi)}|f(z)|\frac{1}{\nu(T(z))}d\mu(z)\rt)Ng(\xi)\omega(\xi)|d\xi|\\
&\succeq \int_{\PD}\int_{\Gamma(\xi)}|f(z)||g(z)|\frac{1}{\nu(T(z))}d\mu(z)\omega(\xi)|d\xi|\\
&= \int_{\DD}|f(z)||g(z)|d\mv(z).
\end{align*}
One the other hand,
\begin{align*}
\int_{\PD}\lf(\int_{\Gamma(\xi)}|f(z)|\frac{1}{\nu(T(z))}d\mu(z)\rt)Ng(\xi)\omega(\xi)|d\xi|&=\int_{\PD}\A f(\xi)Ng(\xi)\omega(\xi)|d\xi|\\
&\preceq \|f\|_{\H^p(\omega)}\|g\|_{\H^{q'}(\omega)}.
\end{align*}
Therefore
\[\int_{\DD}|f(z)||g(z)|d\mv(z)\preceq \|f\|_{\H^p(\omega)}\|g\|_{\H^{q'}(\omega)}.\]
Since $f\in\H^p(\omega)$ and $g\in\H^{q'}(\omega)$ are arbitrary, by Lemma \ref{lmm5}, for each $f\in\H^t(\omega),$
\[\int_{\DD}|f(z)|d\mv(z)\leq \|f\|_{\H^t(\omega)},\]
here $t=\frac{pq}{pq-p+q}>1.$
So the function
\begin{align*}
\sigma(\xi)=\frac{1}{\omega(\xi)}\int_{\Gamma(\xi)}\frac{1}{1-|z|}d\mv(z)
=\frac{1}{\omega(\xi)}\int_{\Gamma(\xi)}\frac{1}{1-|z|}\frac{\omega(T(z))}{\nu(T(z))}d\mu(z)
\end{align*}
belongs to $L^{\frac{t}{t-1}}(\omega)=L^{\frac{pq}{p-q}}(\omega)$ by Theorem \ref{b}.
Since $\omega\in A_1$ and $1-|z|\asymp |T(z)|$ whenever $z\in\Gamma(\xi).$ Thus,
\begin{align*}
\sigma(\xi)&=\frac{1}{\omega(\xi)}\int_{\Gamma(\xi)}\frac{1}{1-|z|}\frac{\omega(T(z))}{\nu(T(z))}d\mu(z)\\
&\asymp \frac{1}{\omega(\xi)}\int_{\Gamma(\xi)}\frac{\omega(T(z))}{|T(z)|}\frac{1}{\nu(T(z))}d\mu(z)\\
&\asymp\vv(\xi).
\end{align*}
These imply that $\vv\in L^{\frac{pq}{p-q}}(\omega).$ The proof is complete.
\end{proof}

\end{document}